\documentclass{article}
\usepackage[margin=1.2in]{geometry}
\usepackage{mystyle}
\usepackage{authblk}

\numberwithin{equation}{section}

\title{Adaptive Optimal Transport}

\author[1]{Montacer Essid}
\author[2]{Debra Laefer}
\author[1]{Esteban G. Tabak}
\affil[1]{Courant Institute of Mathematical Sciences, NYU, \protect \\251 Mercer St, New York, NY 10012}
\affil[2]{NYU Center of Urban Science and Progress, \protect \\ 370 Jay St, Brooklyn, NY 11201}

\date{\today}

\begin{document}
\maketitle

\begin{abstract}
An adaptive, adversarial methodology is developed for the optimal transport problem between two distributions $\mu$ and $\nu$, known only through a finite set of independent samples $(x_i)_{i=1..N}$ and $(y_j)_{j=1..M}$. The methodology automatically creates features that adapt to the data, thus avoiding reliance on a priori knowledge of data distribution. Specifically, instead of a discrete point-by-point assignment, the new procedure seeks an optimal map $T(x)$ defined for all $x$, minimizing the Kullback-Leibler divergence between $(T(x_i))$ and the target $(y_j)$. The relative entropy is given a sample-based, variational characterization, thereby creating an adversarial setting: as one player seeks to push forward one distribution to the other, the second player develops features that focus on those areas where the two distributions fail to match. The procedure solves local problems matching consecutive, intermediate distributions between $\mu$ and $\nu$. As a result, maps of arbitrary complexity can be built by composing the simple maps used for each local problem. Displaced interpolation is used to guarantee global from local optimality. The procedure is illustrated through synthetic examples in one and two dimensions.
\end{abstract}


\section{Introduction}

The optimal transport problem consists of finding, from among all transformations $y = T(x)$ that push forward a source distribution $\mu(x)$ to a target $\nu(y)$, the map that minimizes the expected transportation cost:
\begin{equation}
\min_T  \int c\left(x, T(x)\right) \mu(x) \ dx, \quad T_{\#} \mu = \nu,
\label{eq:MOT}
\end{equation}
where $c(x,y)$ is the externally provided cost of moving a unit of mass from $x$ to $y$ \cite{villani2003topics}. The application for which Monge formulated the optimal transport problem was the actual transportation of material between two sites at minimal cost \cite{monge1781memoire}. Two centuries later, starting with Kantorovich and Koopmans \cite{kantorovich1942v}, the problem was relaxed from maps to couplings, and applied to more general matching problems, such as matching supply and demand or positions and employees. More recently, the optimal transport problem has become a central tool in many computer and data science applications, as well as in analysis and partial differential equations. Among the many applications for which optimal transport could be used, the particular one that drove the methodology proposed in this article is change detection, for which one seeks a correspondence between two point clouds (from remote sensing data -- either imagery or laser scanning) in order to identify differences between them.

The numerical solution of optimal transportation problems has been an active area of research for some years. When the two measures $\mu$ and $\nu$ have discrete support, the relaxation of optimal transport due to Kantorovich \cite{kantorovich1942v} becomes a linear programming problem, which can be solved effectively for problems of small and medium size. When the size of the problem grows, its solution can be accelerated significantly through the addition of an entropic regularization and a Sinkhorn-type iterative algorithm \cite{Cuturi,peyre2018computational}. This regularized problem, both in the discrete and the continuous versions, is equivalent to the Schr\"odinger bridge \cite{leo2,CGP2016}. When the space underlying the two measures $\mu$ and $\nu$ is continuous and the distributions are known in closed form, one can --in small dimensional problems-- discretize them on a grid or a graph before applying these techniques.
Then their solution provides a point-by-point assignment between the source and the target measures. 

However, in most data science applications, the distributions underlying the source and/or target samples are unknown. Moreover, those samples are often embedded in a high dimensional space, and the data are relatively scarce. Density estimation techniques using this scarce data will yield a poor representation of the source and target measures. Hence the transport map or transference plan provided by these techniques will be either inaccurate or highly over-fitted, which leads to a very poor predictive power for the target of new sample points from the source.

In order to provide a more flexible framework for data science applications, sample-based techniques to solve the OT problem were developed in \cite{tabak2017conditional, kuang2017sample, tabak2018explanation}. A central question to address when posing sample-based OT problems is the meaning of the push-forward condition $T_{\#} \mu = \nu$ when $\mu$ and $\nu$ are only known through samples $\left\{x_i\right\}$, $\left\{y_j\right\}$. In the formulations in \cite{tabak2017conditional, kuang2017sample, tabak2018explanation}, this condition was relaxed to the equality of the empirical means of a pre-determined set of functions or ``features'' over the two sample sets; a relaxation that appears naturally in the dual formulation of the problem. This raises the feature selection problem of finding the set of features best suited to each application. The associated challenges are particularly apparent in the change detection problem, where elements in two point clouds may differ for instance in size, color, shape, data distribution or location, may be large or small, may have appeared, disappeared, have been displaced, deformed, broken, consolidated\ldots  Thus the development of a robust, application-independent feature-selection methodology is far from trivial.

The methodology proposed in this article incorporates feature selection into the formulation of the optimal transport problem itself, through an adversarial approach. This involves three main steps:
\begin{enumerate}
\item Borrowing from the methodology developed in \cite{kuang2017sample}, we subdivide the transportation problem between $\mu$ and $\nu$ into finding $N$ local maps $T_t$ pushing forward $\rho_{t-1}$ to $\rho_{t}$, with $\rho_0 = \mu$ and $\rho_{N}=\nu$. The global map $T$ results from the composition of these local maps: $T = T_{N} \circ T_{N-1} \circ \ldots \circ T_1$, and global optimality is guaranteed by requiring that the $\rho_t$ are McCann's displacement interpolants \cite{mccann1997convexity} between $\mu$ and $\nu$.
This decomposition achieves two goals:
\begin{itemize}
  \item Because every pair of successive $\rho_t$ are close to each other, the corresponding maps $T_t$ are close to the identity, which is the gradient of the strictly convex function $\frac{1}{2} \|x\|^2$. This permits relaxing the requirement that $\phi_t$ be convex in the optimality condition $T_t = \nabla \phi_t$ for the standard quadratic transportation cost. 
  \item Arbitrarily complex maps $T$ can be built through the composition of quite simple maps $T_t$. Thus, the maps over which to optimize each local problem can be reduced to a suitable family depending on just a handful of parameters. 
\end{itemize}

\item We formulate the push-forward condition $T_{t\: \#} \rho_{t-1} = \rho_t$ not in terms of the empirical expectation of features but as the minimization of the relative entropy between $T_{t\: \#} \rho_{t-1}$ and  $\rho_t$. One advantage of this formulation is that it is a natural relaxation of the push-forward condition when $T_t$ is restricted to a small family of maps, which renders impossible the achievement of a perfect match between $T_{t\: \#} \rho_{t-1}$ and  $\rho_t$.

\item We use a variational characterization of the relative entropy, as the maximizer of a suitable functional over functions $g(x)$. This formulation has three critical properties:
\begin{enumerate}
\item Since the variational characterization involves expected values of functions over $\rho_{t-1}$ and  $\rho_t$, it can be immediately extended to a sample-based scenario, thereby, replacing those expected values by empirical means.
\item Replacing ``all'' functions $g(x)$ by a suitable family of functions provides a natural relaxation in the presence of finite sample sets. We show that, unlike the maps $T_t$, which produce the global map $T$ via composition, it is the sum of the functions $g_t$ that approximates the global $g$. Moreover, we prove that, if the families of $T_t$ and $g_t$ are built through the linear superposition of a predetermined set of functions, we recover the solution in \cite{kuang2017sample}.

\item Each local problem has now been given a minimax formulation (minimize over $T$, maximize over $g$.) This has a natural adversarial interpretation: while the ``player'' with strategy $T$ seeks to minimize the discrepancies between $T_{\#} \rho_{t-1}$ and $\rho_t$; its adversary with strategy $g$ develops features to prove that the two distributions have not been perfectly matched. This provides the desired adaptability: the user does not need to provide features adapted to the problem in hand, as these will emerge automatically from its solution. This facilitates applications across a broad range of problems, including problems with significant features at various, possibly unknown scales.

\end{enumerate}

\end{enumerate} 

This paper is organized as follows. After this introduction, section \ref{sec:AOT} describes the methodology and its theoretical underpinning. Subsection \ref{subsec:adv} introduces the variational characterization of the relative entropy that the algorithm uses and concludes with the sample-based minimax formulation of the local optimal transport problem. Subsection \ref{subsec:conn} shows that, when the functions $g$ and potentials $\phi$ are drawn from finite-dimensional linear functional spaces, the solution to the problem agrees with  the one obtained in \cite{kuang2017sample} with pre-determined features. Subsection \ref{subsec:duality} proves that the order of minimization and maximization does not matter --that is, that there is no duality gap-- and explains the intuition behind the adversarial nature of the game, by detailing how each player reacts to the other's strategy. Subsection \ref{globalalgo} integrates the local algorithm just described into a global algorithm for the full optimal transport between $\mu$ and $\nu$.

Section \ref{sec:algorithm} details the algorithm further. Subsection \ref{subsec:functionalspace} specifies the functional spaces chosen for $g$ and $\phi$, subsection \ref{subsec:localAlgorithm} the procedure used for solving the minimax problem, and subsection \ref{subsec:penalty} the additional penalization terms required for the non-linear components of the functional spaces. Finally, section \ref{sec:experiments} performs some illustrative numerical experiments, applying the new methodology to synthetic low-dimensional data. The focus of these experiments is to display in action, in easy to visualize scenarios, the adversarial nature of the formulation. 


\section{Adaptive optimal transport}
\label{sec:AOT}

\subsection{Formulation of the problem: an adversarial approach}
\label{subsec:adv}

We are given two sample sets $(x_i)_{i=1,..,n}$, $(y_j)_{j=1,..,m}$ $\subset \R^d$ with $n$ and $m$ sample points respectively, independent realizations of two random variables with unknown distributions $\mu$ and $\nu$.
Both distributions are assumed to be absolutely continuous with respect to the Lebesgue measure on $\R^d$ and have finite second order moments. By a slight abuse of notation, we will identify the measures and their densities.

In this case, Brenier's theorem \cite[p. 66]{villani2003topics} guarantees the existence of a map $T$ 
pushing forward $\mu$ to $\nu$ and minimizing the transportation cost
\begin{equation}
\int \left\|T(x) - x\right\|^2 \mu(x) \ dx.
\end{equation}

From the samples provided, we seek a map $T$ that would perform the transport well when applied to other independent realizations of the unknown distributions $\mu,\nu$.
We can assume that the source and target distribution are close:
\begin{remark}
Solving the problem for nearby distributions is the building block of a general procedure for arbitrary distributions and for finding the Wasserstein barycenter of distributions \cite{kuang2017sample}. This more general procedure is presented in Section \ref{globalalgo}.
\end{remark}

The OT problem has two main ingredients: the push-forward condition that $(T(x_i))$ and $(y_j)$ have the same distribution and the minimization of the cost.
\begin{remark}
For the quadratic cost, the optimal solution is the gradient of a convex function $\phi(x)$, $y = T(x) = \nabla\phi(x)$, a convenient characterization. More general cost functions of the type $\ell(x-y)$ would only require modifying $\nabla \phi$ into $x - \nabla \ell^*(\nabla \phi)$, where $\ell^*$ represents the Legendre-Fenchel transform of the strictly convex function $\ell$, in the algorithm presented below.
\end{remark}
 In \cite{kuang2017sample}, the push-forward condition was formulated in terms of the equality of the empirical expected values of a pre-determined set of feature functions. Instead, we propose a broader and adaptive formulation, in terms of the relative entropy between the two distributions. This introduces some significant improvements:
\begin{enumerate}

\item Of the two characterizations of equality of distributions: that all test-functions within a broad enough class agree and that their relative entropy vanish, the latter is far more succinct and easier to enforce. 

\item Replacing ``all'' test functions by a finite set, though a sensible approximation in the presence of finite sample-sizes, leads to questions of robustness and feature selection. To address this, we will use a variational characterization of the relative entropy, which automatically selects the ``best'' features within a given class.

\item For finite sample sets, one would expect the empirical expected values of test functions on the two distributions to agree only in a statistical sense, so requiring their strict equality is somewhat artificial. By contrast, in the new formulation, rather than requiring the relative entropy to vanish, which may be unrealistic for finite sample-sizes and a limited family of maps $T$, we seek to minimize it. 

\end{enumerate}  

\begin{definition}
For two probability measures $\rho, \nu \in P(\R^d)$, the \emph{Kullback-Leibler divergence} of $\rho$ with respect to $\nu$ --also called their relative entropy-- is defined as
\begin{equation}
D_{KL}(\rho||\nu) = \int \log\left( \frac{d\rho}{d\nu}\right)d\rho
\end{equation}
if $\rho$ is absolutely continuous with respect to $\nu$ ($\rho \ll \nu$), and $+\infty$ otherwise.
\end{definition}

Solving the optimal transport problem is equivalent to minimizing a Kullback-Leibler divergence, as the following proposition shows:

\begin{proposition}\label{prop:KLmin}
Let $\mu,\nu \in P(\R^d)$, with $\mu$ absolutely continuous with respect to the Lebesgue measure $m$ on $\R^d$.

Let $\mathcal{C}$ be the set of convex functions from $\R^d \to \R$.

Define the minimization problem
\begin{equation}\label{eq:KLOPT}\tag{KLopt}
\inf_{\phi \in \mathcal{C}} D_{KL}(\nabla \phi_{\#} \mu ||\nu) 
\end{equation}
where $\nabla \phi_{\#} \mu (A) = \mu((\nabla \phi)^{-1}(A))$ \footnote{$\nabla \phi$ is well defined $m$-a.e. by Theorem 25.5 \cite{rockafellar1970Convex}, and hence $\mu$-a.e.}, for any Borel measurable set $A$.

Then there exists a unique minimizer $\phi$ (up to zero measure sets), which coincides with the minimizer of the 2-Wassertein distance between $\mu$ and $\nu$:
\[
\phi = \arg \inf_{\psi \in \mathcal{C}}D_{KL}(\nabla \psi_{\#} \mu ||\nu)
\]
and
\[
W_2^2(\mu,\nu) = \int |\nabla \phi(x)-x|^2 d\mu(x).
\]
\end{proposition} 

\begin{proof}
By Brenier's theorem, there exists a unique minimizer $\phi$ (up to zero measure sets) for the 2-Wassertein problem. The potential $\phi$ is a proper lower semi-continuous convex function and $\nabla \phi_{\#} \mu = \nu$.
One easily sees that $\phi$ is the minimizer for the Kullback-Leibler divergence optimization problem \eqref{eq:KLOPT}, since for any measure $\rho \in P(\R^d)$ one has, 
\begin{equation}\label{eq:DKL}
D_{KL}(\rho||\nu) \geq 0
\end{equation}
with equality if and only if $\rho = \nu$ almost everywhere (in $\nu$).

Inequality \eqref{eq:DKL} is easy to prove: if $\rho$ is not absolutely continuous w.r.t. $\nu$, the Kuklback-Leibler divergence is infinite, so the statement is true. Otherwise, we have
\begin{align*}
D_{KL}(\rho ||\nu) &= \int \log\left( \frac{d\rho}{d\nu}\right)d\rho\\
 &= \int \log\left( \frac{d\rho}{d\nu}\right) \frac{d\rho}{d\nu} d\nu \\
 &\geq \left( \int \frac{d\rho}{d\nu} d\nu \right) \log\left(\int \frac{d\rho}{d\nu} d\nu \right) = 0,
\end{align*}
where we used Jensen's inequality and the convexity of $x \mapsto x \log(x)$.

So the equality $D_{KL}(\rho||\nu)=0$ will hold if and only if Jensen's inequality becomes an equality, i.e. if and only if $\frac{d\rho}{d\nu} \equiv 1$, or $\rho = \nu$.

In particular, the solution to the optimal transport problem satisfies $\nabla \phi_{\#} \mu = \nu$. Hence
\[
D_{KL}(\nabla \phi_{\#} \mu||\nu) = 0,
\]
which shows that $\phi$ is a minimizer of the optimal transport problem and \eqref{eq:KLOPT}.

As for uniqueness, let $\phi_1, \phi_2 \in \mathcal{C}$ be two minimizers. Then $\nabla {\phi_1}_{\#} \mu = \nabla {\phi_2}_{\#} \mu = \nu$ from the statement above. By Brenier's theorem, they both solve the quadratic cost optimal transportation problem, which has a unique solution up to zero measure sets.
\end{proof}

Recently there has been a push in machine learning to replace the Kullback-Leibler divergence by Wasserstein distances in order to penalize differences in data sets \cite{frogner2015learning, peyre2018computational}. Unlike the Kullback-Leibler divergence, the Wasserstein distance defines a proper distance, enjoys regularity and symmetry properties, and is computationally tractable. 
Nonetheless, the Kullback-Leibler divergence is well suited to measure the dissimilarities between measures that we are trying to detect. In particular, the asymmetry between the two measures under the Kullback-Leibler divergence is well within the spirit of the problem, as we seek a convex function $\phi$ that makes the transported distribution $\nabla \phi_{\#} \mu$ indistinguishable from the target reference $\nu$.
Also, as we shall see, the minimization of the relative entropy captures the differences between the two sample sets far more deftly than does a predefined finite set of test functions.

Thus, the biggest drawback in using the Kullback-Leibler divergence appears to be the difficulty in its numerical evaluation, particularly when we do not have access to a closed form expression for $\mu$ and $\nu$, but merely to a finite set of independent samples from each of these distributions.
One could resort to density estimation techniques \cite{sheather1991reliable,silverman2018density} to approximate $\mu$ and $\nu$ and then proceed to numerical integration.
Instead, we use a variational characterization of the Kulback-Leibler divergence of $\rho$ with respect to $\nu$, in the form of a sample-friendly expression :
\begin{proposition}\label{prop:KLvar}
Let $\rho, \nu \in P(\R^d)$. Then
\[
D_{KL}(\rho||\nu) = 1 + \sup_{g} \left\lbrace \int g d\rho - \int e^g d\nu \right\rbrace
\]
\end{proposition}
over all Borel measurable functions $g: \R^d \to \R$.
\begin{proof}
If we do not have $\rho \ll \nu$, there exists a set $A \subset \R^d$ such that $\rho(A)>0$ and $\nu(A) = 0$.
Then 
\[
1 +\sup_{g} \left\lbrace \int g d\rho - \int e^g d\nu \right\rbrace 
\]
is infinite, as it can be made arbitrarily large by picking functions of the type $g = c \mathbbm{1}_A$, $c \in \R$.
$D_{KL}(\rho||\nu)$ is also infinite in this case. Hence their values agree.

When $\rho \ll \nu$, notice that for $\nu$-almost every $x \in \R^d$, 
\[
g \in \R \mapsto g \frac{d \rho}{d \nu}(x) - e^g 
\]
is concave and maximized for $g(x) = \log \left( \frac{d\rho}{d\nu}(x) \right)$ (note that the Radon-Nikodym derivative $\frac{d\rho}{d\nu}$ is non-negative, $\nu$-a.e.). 

Thus, for almost every $x \in \R^d$ and any choice of $g(x) \in \R$, we have:
\[
1 + g(x)\frac{d \rho}{d \nu}(x) - e^{g(x)} \leq 1 + \frac{d\rho}{d\nu}(x)  \left[ \log \left( \frac{d\rho}{d\nu}(x) \right) - 1 \right]
\]
with equality if and only if $g(x) = \log \left( \frac{d\rho}{d\nu}(x) \right)$.

Integrating over the measure $\nu$ yields
\[
1 + \int_{\R^d} \left(g(x)\frac{d \rho}{d \nu}(x) - e^{g(x)} \right) d\nu(x) \leq \int_{\R^d}  \log \left( \frac{d\rho}{d\nu}(x) \right) d\rho(x) = D_{KL}(\rho||\nu)
\]
and, thus, one has 
\[
1 + \sup_{g} \left\lbrace \int_{\R^d}  g(x) d \rho(x) - \int_{\R^d} e^{g(y)} d \nu(y) \right\rbrace = D_{KL}(\rho||\nu) 
\]
since we have equality for 
\[
g = \log \left( \frac{d\rho}{d\nu} \right) \quad \text{ on the support of } \nu.
\]
\end{proof}

\begin{remark}
\begin{enumerate}
\item The variational reformulation of the Kullback-Leibler divergence is a consequence of the convexity of $x \mapsto - \log(x)$.
Indeed, computing its Legendre-Fenchel transform twice yields: 
\[
- \log(x) = \sup_{y<0} \left\lbrace x y + 1 -  \log\left(-\frac{1}{y} \right) \right\rbrace  = \sup_{g \in \R} \{ g - x e^g \} + 1 
\]
This approach extends to a broader set of f-divergences, yielding similar variational formulations, see \cite{nguyen2010estimating} and \cite{nowozin2016f}.

\item A very similar variational formulation was developed in \cite{suzuki2008approximating} to estimate the likelihood of two samples being generated from independent sources. 

\item Note that the variational formulation represented above is very similar to the Donsker-Varadhan \cite{donsker1975asymptotic} formula:
\[
\sup_{g} \left\lbrace \int_{\R^d}  g(x) d \rho(x) - \log \left(\int_{\R^d} e^{g(y)} d \nu(y) \right) \right\rbrace
\]
Indeed, $\log(x) \leq x -1$ yields:
\[
\sup_{g} \left\lbrace \int_{\R^d}  g(x) d \rho(x) - \int_{\R^d} e^{g(y)} d \nu(y)  \right\rbrace + 1
 \leq \sup_{g} \left\lbrace \int_{\R^d}  g(x) d \rho(x) - \log \left(\int_{\R^d} e^{g(y)} d \nu(y) \right) \right\rbrace
\]
and equality is achieved for the same maximizer $g = \log\left(\frac{d\rho} {d\mu}\right)$, if $\rho \ll \nu$ (otherwise, they are both infinite). The formula in Proposition~\ref{prop:KLvar} can be considered as a linearization of the Donsker-Varadhan formula, easier to implement numerically.
\end{enumerate}
\end{remark}

Given two random variables $Z \sim \rho$ and $Y \sim \nu$ with $\rho \ll \nu$, we can equivalently express the formula in Proposition \ref{prop:KLvar} as:
\[
D_{KL}(\rho||\nu) = 1 + \max_{g} \left\lbrace \E[g(Z)] - \E[e^{g(Y)}] \right\rbrace
\]

If instead, we are given  \emph{independent samples} $z_1,...,z_n$ of $Z$, and $y_1,..,y_m$ of $Y$, we can approximate the above reformulation by its empirical counterpart:
\[
D_{KL}(\rho||\nu) \approx 1 + \max_{g} \left\lbrace  \frac{1}{n} \sum_i g\left(z_i\right) - \frac{1}{m} \sum_j e^{g(y_j)} \right\rbrace
\]
where the maximization is sought over a suitable class of functions $g$. Theorem 1 of \cite{nguyen2010estimating} shows that if this class of functions
\begin{enumerate}
\item contains the optimizer $g^* = \log \left( \frac{d \rho}{d \nu} \right)$,
\item satisfies the envelope conditions \cite{nguyen2010estimating}[16a, 16b] (e.g. $g$ is bounded),
\item satisfies the entropy conditions \cite{nguyen2010estimating}[17a, 17b] (e.g. Sobolev spaces $\mathcal{W}^{k,2}$ on a compact space),
\end{enumerate}
then we have Hellinger consistency of this estimator, that is 
\begin{equation}\label{eq:Hellinger}
\int \left( \sqrt{\exp(g^*)} - \sqrt{\exp(g_{n,m})} \right)^2 d\nu \xrightarrow[n,m \to +\infty]{} 0
\end{equation} 
where $g_{n,m} = \arg\max \left\lbrace  \frac{1}{n} \sum_i g\left(z_i\right) - \frac{1}{m} \sum_j e^{g(y_j)} \right\rbrace$.

%

We deduce from Propositions \ref{prop:KLmin} and \ref{prop:KLvar} the following reformulation of the optimal transport problem between $\mu$ and $\nu$, under a quadratic cost, expressed as a minimax problem:
\begin{problem}[Minimax reformulation]\label{minimax}
\begin{equation*}
\min_{\phi} \max_{g} L[\phi,g] \equiv \E\left[g(\nabla \phi(X))\right] - \E\left[e^{g(Y)}\right]
\end{equation*}
\end{problem}
Note that the \emph{Lagrangian} $L$ is concave in the maximization variable $g$, but not necessarily convex in the minimization variable $\phi$.

The sample-based version of Problem \ref{minimax} is given by:
\begin{problem}[Sample based minimax reformulation]\label{sampleMinimax}
\begin{equation*}
\min_{\phi} \max_{g} L[\phi,g] \approx \min_{\phi} \max_{g} \left\lbrace \frac{1}{n} \sum_i g\left(\nabla\phi(x_i) \right) - \frac{1}{m} \sum_j e^{g(y_j)} \right\rbrace
\end{equation*}
\end{problem}
over \emph{suitable} function spaces for $\phi(x)$ and $g(y)$, as detailed in Section \ref{sec:algorithm}.

This is an adversarial setting, in which the player with strategy $\phi$ attempts to minimize the discrepancies between the distributions underlying the sample sets $\left\{\nabla \phi(x_i)\right\}$ and $\left\{y_j\right\}$, while the player with strategy $g$ attempts to show that the two distributions are in fact different. Thus $g$ would point to those areas where the two distributions differ the most, and $\phi$ would correct those discrepancies. We will see this competition in action in the examples in section \ref{sec:experiments}.

This saddle point optimization problem is reminiscent of the ones encountered in the Generative Adversarial Networks (GAN) literature \cite{nowozin2016f}. Broadly speaking, a GAN learns how to generate a sample from an unknown distribution. To do so, a two-player game is introduced; a parameterized \emph{generator} $Q$ aims to produce samples as `close' as possible to the samples in the training set. This is quantified by the use of an f-divergence (e.g. Kullback-Leibler, Jensen-Shannon, or `GAN' divergence), which is given a variational formulation in the exact same way as it is done in Proposition~\ref{prop:KLvar}. This in turns introduces a \emph{discriminator}, whose role is to prove that the generator has not done the right job.

Formulated as such, our optimization problem is quite similar to a GAN. Indeed, the generator $Q$ is a distribution which is usually induced by the pushforward of a generic distribution (e.g. standard Gaussian) by a map $T$. This map, as well as the discriminator, are calibrated using neural networks. This is well within the spirit of the method we use to generate the optimal transport map, as well as the function $g$ (see Section \ref{globalalgo}).

The main differences with the algorithm presented in \cite{nowozin2016f} and ours are:
\begin{enumerate}
\item Our map $T$ is restricted to the form $\nabla \phi$ where $\phi$ is convex, in order to solve the quadratic Wasserstein problem. To our knowledge, there are no restrictions on the map in the GAN problem,
\item We use a variational formulation of the Kullback-Leibler divergence instead of the `GAN' divergence,
\item Instead of using a batch gradient descent for the optimization algorithm, we proceed to what we call `implicit gradient descent', which is described in Section~\ref{subsec:localAlgorithm}.
\item Although our method of generating the map $T$ and `discriminant' $g$ proceed to a sum or composition of many non-linear maps, we do not directly use neural networks.
\end{enumerate}

\subsection{Connection with the pre-determined features case}
\label{subsec:conn}
In \cite{kuang2017sample}, a set of `features' $f_1,..,f_K$ serve as test functions to evaluate the statement $\rho = \nu$ for $\rho,\nu \in P(\R^d)$, when we only have sample points $(z_i)_{i=1,..,n}$ and $(y_j)_{j=1,..,m}$ generated from $Z \sim \rho$, $Y \sim \nu$.
As in \cite{kuang2017sample}, we will assume that $\mu,\nu$ are `close'. The general case with more distant measures can be reduced to the solution of many local problems, as shown in Algorithm \ref{algo:TOT} below, also borrowed from \cite{kuang2017sample}.

\begin{definition}
The samples $(z_i)_{i=1,..,n}$ and $(y_j)_{j=1,..,m}$ generated from random variables $Z \sim \rho$, $Y \sim \nu$ are equivalent for the set of features $f_1,..,f_K$ if
\[
\frac{1}{n} \sum_{i=1}^n f_k(z_i) = \frac{1}{m} \sum_{j=1}^m f_k(y_j), \quad \forall k=1,..,K
\]
\end{definition}

The definition above is a relaxation of the equivalence $\mu = \nu \Leftrightarrow \E[f(Z)] = \E[f(Y)]$ for all test functions $f \in C_b(\R^d)$.
Then solving the transport problem between the samples $(x_i)$ and $(y_j)$ is reduced to finding a map $T$ such that $(T(x_i))_i$ is equivalent to $(y_j)$, for the features $f_1,..,f_K$. 

In \cite{kuang2017sample}, $T$ is chosen to be of the type :
\[
T(x) = \nabla \phi(x) = x + \sum_{k} \alpha_k \nabla \phi_k(x)
\]
for some pre-determined functions $\phi_1,..,\phi_K$ and constants $\alpha_1,...,\alpha_K$. In fact, the potentials $\phi_k$ adopted in \cite{kuang2017sample} agree with the features $f_k$, but our proposition below applies to more general choices.
It shows that the procedure to solve the sample-based optimal transport problem with pre-determined features is a particular instance of Problem \ref{sampleMinimax}. A specific choice of functional space for $g$ will yield this result. Before introducing it, we need a set of compatibility conditions for the choices of possible $\phi$ and $g$.

\begin{definition}
The features $f_k, \: k=1,..,K$ are said to be compatible with the potentials $\phi_k, \: k=1,..,K$ for the sample $(x_i)_{i=1,..,n}$, if the matrix $C \in \R^{K \times K}$ defined as  
\[
C_{kk'} = \frac{1}{n}  \sum_{i=1}^n \nabla \phi_{k}(x_i) \cdot \nabla f_{k'}(x_i)
\]
is non-singular.
\end{definition}

This compatibility assumption essentially guarantees the non-degeneracy of the choice of functions, as it restricts the average displacement to affect the features in an independent fashion. 
It can be summarized by the requirement that $C = \E[J_{\phi} J_f^{\top}]$ is non-singular, where $J_{\phi}, J_f$ are the Jacobian matrices of $\phi,f$.

\begin{proposition}\label{predeterminedFeatures}
Given a compatible set of features $f_1,..,f_K$ and potentials $\phi_1,..,\phi_K$ for the sample $(x_i)_{i=1,..,n}$, consider Problem \ref{sampleMinimax} using the functional spaces:
\[
g(z) = \sum_{k=1}^K \beta_k f_k(z), \quad \phi(x) = \frac{|x|^2}{2} + \sum_{k=1}^K \alpha_k \phi_k(x) 
\]
for $\beta \in \R^K$, $\alpha \in \R^K$ in a small-enough neighborhood of zero.

Then the optimizer $\phi$ of Problem \ref{sampleMinimax} for two sample sets close to each other solves the sample-based optimal transport problem with predetermined features; meaning that $(\nabla \phi(x_i))$ is equivalent to $(y_j)$ for the features $f_1,..,f_K$. :
\[
\frac{1}{n} \sum_{i=1}^n f_k(\nabla \phi(x_i)) = \frac{1}{m} \sum_{j=1}^m f_k(y_j), \quad \forall k=1,..,K
\]

\end{proposition}
\begin{proof}
The Lagrangian $L$ as a function of $\alpha,\beta$ is given by 
\begin{equation*}
\frac{1}{n} \sum_i \left[\sum_{k=1}^K \beta_k f_k\left(x_i + \sum_{l=1}^K \alpha_l \nabla \phi_l(x_i) \right) \right] 
  - \left[\frac{1}{m} \sum_j e^{\sum_{k=1}^K \beta_k f_k(y_j)}\right]
\end{equation*}
Taking the first order conditions at optimality yields:
$$ \nabla_{\alpha}  L = C(\alpha) \beta, \quad \text{ where } C(\alpha)_{kk'} =  \frac{1}{n} \sum_i \left[\nabla \phi_k(x_i)\cdot\nabla f_k\left(x_i + \sum_{l=1}^K \alpha_l \nabla \phi_l(x_i) \right)\right], $$

Since $\alpha$ is in a neighborhood of zero, the matrix $C(\alpha)$ is a small perturbation of the non-singular matrix $C$. Since features and potentials are compatible, the matrix $C$ is non-singular, and, thus, $C(\alpha)$ is non-singular itself. Hence

$$ \nabla_{\alpha} L = 0 \Rightarrow \beta = 0. $$

Moreover, the second optimality condition evaluated at $\beta = 0$ yields $\forall k$:
$$ \partial_{\beta_k} L =  \frac{1}{n} \sum_i f_k\left(x_i + \sum_{l=1}^K \alpha_l \nabla \phi_l(x_i) \right) 
- \frac{1}{m} \sum_j f_k(y_j)$$

Hence $\nabla_{\beta} L = 0$ at $\beta = 0$ implies that 
\[
\frac{1}{n} \sum_i f_k\left(x_i + \sum_{l=1}^K \alpha_l \nabla \phi_l(x_i) \right) 
= \frac{1}{m} \sum_j f_k(y_j)
\]
Notice that the closeness of the two sample sets and the compatibility between the potential and features guarantee that this problem has a solution with a small $\alpha$ (in fact, this can be taken as a feature-dependent characterization of what it means for two sample sets to be close to each other).
This result means that the empirical expected values of the $f_k$ agree on $\left\{T(x_i)\right\}$ and $\left\{y_j\right\}$, i.e. the samples are equivalent for the features $f_1,...,f_K$. Hence $T=\nabla \phi$ solves the sample-based optimal transport problem with pre-determined features.
%
%
\end{proof}

Note that we are restricting the maps $\nabla \phi$ to be `small' perturbations of the identity, by choosing $\alpha$ in a neighborhood of $0$. This is because the optimal transport procedure will only be applied to measures or samples that are `close' to each other. 

In this paper, we will allow $g$ to be more general than a simple linear combination of features, thus greatly expanding the procedure in \cite{kuang2017sample}. This added flexibility yields better adaptability to the most important characteristics of the data.

\subsection{Duality}
\label{subsec:duality}
\subsubsection{No duality gap}
Given the Lagrangian $L$ introduced in Problem \ref{minimax}, the primal objective functional to minimize is, according to Proposition \ref{prop:KLvar}:
\begin{equation}\label{eq:primal}
D[\phi] = \max_g L[\phi,g] = D_{KL} \left( \nabla \phi_{\#}\mu || \nu \right) - 1
\end{equation}
The proof in Proposition \ref{prop:dualitygap} shows that the dual objective functional to be maximized is:
\begin{equation}\label{eq:dual}
d[g] = \min_{\phi} L[\phi,g] = \left(\min_{y \in \R^d} g(y) \right) - \mathbb{E} \left[ e^{g(Y)} \right] 
\end{equation}
A desired property of the adversarial game, defined by the formulation in Problem \ref{minimax}, is the absence of an irreversible advantage or penalty a player gets from playing first.
In other words we do not want a duality gap. This is the content of the following proposition:
\begin{proposition}[Absence of duality gap]\label{prop:dualitygap}
\[
\min_{\phi} \max_{g} L[\phi,g] = \min_{\phi} D[\phi]  = \max_{g} d[g] =  \max_{g} \min_{\phi} L[\phi,g]
\]
\end{proposition}
\begin{proof}
From Proposition \ref{prop:KLmin}, we know that 
\[
\min_{\phi} D_{KL}(\nabla \phi_{\# }\mu||\nu) = 0
\]
with the minimizer reached for the solution of the transport problem. 

Hence we get in Equation \eqref{eq:primal}
\[
\min_{\phi} \max_{g} L[\phi,g] = \min_{\phi} D[\phi] = -1
\]

On the other hand, maximizing Equation \eqref{eq:dual} yields:
\[
 \max_{g} \min_{\phi} L[\phi, g] = \max_g \left\lbrace \min_{\phi}  \mathbb{E}[g(\nabla \phi(X))] - \mathbb{E}\left[e^{g(Y)}\right]  \right\rbrace
\]
Note that the inner minimum is reached for the convex function $\phi(x) = y_{min} \cdot x$ where $\min_{y} g(y) = g(y_{min}) \equiv g_{min}$.

In the case where the minimum of $g$ is not reached, take a minimizing sequence $y_{min}^{n}$ such that $g(y_{min}^n) \to \inf_{y \in \R^d} g(y) \equiv g_{min}$. Then a minimizing sequence for the inner minimum in $\phi$ is given by $\phi^n(x) = y_{min}^n \cdot x$.

In both cases,
\[
\min_{\phi}  \mathbb{E}[g(\nabla \phi(X))] = g_{min}
\]
We are, thus, left with maximizing the dual problem
\[
\max_g d[g] = \max_g \left\lbrace g_{min} - \mathbb{E} \left[ e^{g(Y)} \right]  \right\rbrace
\]
Since $\mathbb{E}\left[e^{g(Y)}\right] \geq e^{g_{min}}$, we can always choose $g$ to be the constant function $g_{min}$. We are then left with maximizing
\[
\max_{g_{min}} g_{min} - e^{g_{min}}
\]
which is achieved for $g \equiv g_{min} = 0$. Hence we also have that 
\[
\max_g d[g] = \max_{g} \min_{\phi}L(\phi, g)  = -1
\]
\end{proof}
%
%

\subsubsection{An adversarial view of duality}
The optimality conditions for the minimax problem are given by
\[
\begin{cases}
\nabla \phi \text{ moves mass to where } g \text{ is smallest } \\
g(y) = \log \left( \frac{\nabla \phi_{\#} \mu (y)}{ \nu(y)} \right)
\end{cases}
\]
Examining the primal and dual problems in light of these conditions explains the behavior of the competing players $\phi$ and $g$:
\begin{itemize}
\item Given a function $g$, $\phi$ will try to move mass from the areas where $g$ is large (i.e. $\nabla \phi_{\#} \mu (y) \ge \nu(y)$) to those where $g$ is small (i.e. $\nabla \phi_{\#} \mu (y) \le \nu(y)$). Following this strategy allows this player to minimize the impact of $g$ on the Lagrangian.
\item Given a function $\phi$, $g$ will adapt to get closer to the function $\log \left( \frac{\nabla \phi_{\#} \mu(y)}{ \nu(y)} \right)$, which is large where mass is lacking ($\nabla \phi_{\#} \mu (y) \ge \nu(y)$) and vice-versa. Following this strategy, allows the second player to increase the Lagrangian by focusing on those areas where the push-forward condition has not been fully achieved.
\end{itemize}
The game concludes when $g$ becomes constant (necessarily 0) on the support of the distributions. Then $\phi$ does not need to move mass anymore, as it then receives no new directive from $g$.

\subsection{Global algorithm} \label{globalalgo}
%
One could attempt to directly use a procedure based on Problem \ref{sampleMinimax} to solve the OT problem for any samples $(x)_i$ and $(y)_j$.
Such direct approach, however, would not be universally efficient for the following reasons:
\begin{itemize}
\item If the distributions underlying $(x)_i$ and $(y)_j$ are considerably different, one would require a very rich family of potentials to build a $\phi$ that can perform an accurate transfer.
\item One would also require a rich functional space from which to draw $g$ in order to properly characterize all significant differences in the two data samples.
\item Depending on the parametrization of $\phi$ and $g$, the Lagrangian can be non-convex in the variables parametrizing $\phi$, and non-concave in the variables parametrizing $g$. With distributions that are far apart, this could make the numerical solution depend on the initialization of those parameters.
\item The condition that $\phi$ is a convex function is typically hard to enforce. For nearby distributions, on the other hand, it is satisfied automatically, as $\phi(x)$ is close to the convex potential $\frac{1}{2} \|x\|^2$ corresponding to the identity map. 
\end{itemize}

For these reasons, we will solve multiple local optimal transport problems, instead of one global one.
More precisely, we will apply Algorithm \ref{algo:TOT}, adapted from Algorithms 2 and 7 in \cite{kuang2017sample}.
\begin{algorithm}
    \caption{Theoretical Global Optimal Transport Algorithm (TGOT)}
    \label{algo:TOT}
    \begin{algorithmic}[0] 
        \Procedure{TGOT}{$\mu,\nu$} 
        	\State $\triangleright$ \textit{Step 1: Initialize intermediate nodes}
        	\State $N \gets$ number of intermediary steps
        	\State $\rho_0 \gets \mu, \quad \rho_T \gets \nu$
        	\For{$t=1,..,N-1$}
        		\State $\rho_t \gets \frac{N-t}{N}\mu+\frac{t}{N}\nu$ \Comment{ or any arbitrary measure} 
        	\EndFor
        	\While{\textit{not converged}}
        		\State $\triangleright$ \textit{Step 2: Forward step}
        		\For{$t=1,..,N$}
        			\State Solve the optimal transport problem between $\rho_{t-1}$ and $\rho_t$, as defined in Problem \ref{minimax}. It yields an `local' optimal map $\nabla \phi_t$.
        		\EndFor
        		\State $\nabla \phi \gets \nabla \phi_N \circ \nabla  \phi_{N-1} \circ \dots \circ \nabla  \phi_1$
        		\State $\triangleright$ \textit{Step 3: Backward step}
        		\For{$t=1,..,N-1$}
        			\State $\rho_t \gets (\frac{N-t}{N}Id + \frac{t}{N} \nabla \phi)_{\#} \mu$
        		\EndFor
        	\EndWhile  	
            \State \textbf{return} $\nabla \phi$
        \EndProcedure
    \end{algorithmic}
\end{algorithm}
Theorem 2.4 in \cite{kuang2017sample} proves the convergence of Algorithm \ref{algo:TOT} to the solution of the OT problem.

In Algorithm \ref{algo:TOT}, the forward step consists of solving multiple, small, optimal transport problems, addressed in Section \ref{subsec:localAlgorithm}. The backward step back-propagates the final sample computed in the forward pass to all the intermediate samples using McCann's displacement interpolants.

This procedure, reminiscent of the neural networks of machine learning with their ``hidden layers''  replaced by local optimal transport problems, introduces several advantages:
\begin{itemize}
\item The global solution will be obtained by composition of the local maps: 
\begin{equation}\label{eq:compMaps}
\nabla \phi = \nabla \phi_N \circ \nabla  \phi_{N-1} \circ \dots \circ \nabla  \phi_1
\end{equation}
Hence one can choose a small family of maps to solve each local optimal transport problem, and still span a rich family of maps for the global displacement.

Note that in our two-player game, we would theoretically have at optimality $T_{\#} \mu = \nu$ and hence the optimal $g$ would be equal to $\log(T_{\#} \mu / \nu) = 0$.
\item If $\rho_t$ and $\rho_{t+1}$ are close, the local OT problem has a solution $\nabla \phi$ that is a small perturbation of the identity, i.e. the gradient of a strictly convex potential.  Starting from the identity, the numerical algorithm will explore a small neighborhood around it. If the solution that we seek is in this neighborhood, convexity will be preserved. 
\end{itemize}

The global algorithm for finding the optimal map between two distributions known through the samples $(x_i)$ and $(y_j)$ is summarized in Algorithm \ref{algo:SBOT}.
\begin{algorithm}
    \caption{Sample Based Global Optimal Transport Algorithm (SBGOT)}
    \label{algo:SBOT}
    \begin{algorithmic}[0] 
        \Procedure{SBGOT}{$(x_i),(y_j)$} 
        	\State $\triangleright$ \textit{Step 1: Initialize intermediate nodes}
        	\State $N \gets$ number of intermediary steps
        	\State $z_0 \gets x, \quad z_N \gets y$
        	\For{$t=1,..,N-1$}
        		\State $z_{t,i} \gets \frac{N-t}{N}x_i+\frac{t}{N} y_{\sigma(i)}$ \Comment{ for some $\sigma:\{1,..,n\} \to \{1,..,m\}$ (or any arbitrary samples)} 
        	\EndFor
        	\While{\textit{not converged}}
        		\State $\triangleright$ \textit{Step 2: Forward step}
        		\For{$t=1,..,N$}
        			\State $z_t \gets SBLOT(z_{t-1},z_t)$ \Comment{see Algorithm \ref{algo:SBLOT}}
        		\EndFor

        		\State $\triangleright$ \textit{Step 3: Backward step}
        		\For{$t=1,..,N-1$}
        			\State $z_t \gets \frac{N-t}{N}x + \frac{t}{N}z_N$
        		\EndFor
        	\EndWhile  	
            \State \textbf{return} $z_N$
        \EndProcedure
    \end{algorithmic}
\end{algorithm}
Algorithm \ref{algo:SBLOT} in Section \ref{sec:algorithm} further details the procedure to solves the sample based local Optimal Transport problem.

\section{Algorithm}\label{sec:algorithm}
In order to complete the description of the algorithm proposed, we need to specify the functional spaces from which $g$ and $\phi$ are drawn and the procedure used for solving the minimax problem for of the Lagrangian $L(g,\phi)$.

\subsection{Choice of functional spaces}\label{subsec:functionalspace}
Since any two consecutive distributions $\mu,\nu$ in the procedure are close to each other, the optimal map is a perturbation of the identity.
The potential $\phi$ will, thus, be chosen in the form:
\begin{equation}
  \phi(x) = \frac{1}{2} \|x\|^2 + \psi(x)
\end{equation}
where $\psi$ has a Hessian with a spectral radius less than $1$.
No such centering is required for $g(x)$, as at optimality $g(x) = \log\left(1 \right) = 0.$

One basic capability that one should require of the functional spaces for $g$ and $\phi$ is that of detecting and correcting global displacements and scaling --not necessarily isotropic -- between two distributions. Thus one should have
\[
\phi(x) = \frac{1}{2} x^{\top}(I+A_0)x + a_1 \cdot x + \phi_{nl}(x)
\]
and
\[
g(z) = \frac{1}{2} z^{\top} B_0 z +b_1 \cdot z  + b_2 + g_{nl}(z),
\]
where $A_0,B_0$ are symmetric matrices in $\R^{d\times d}$, $a_1, b_1$ are vectors in $\mathbb{R}^d$, $b_2 \in \R$ is a scalar, and $\phi_{nl}$ and $g_{nl}$ stand for additional non-linear features discussed below.
The quadratic polynomial in $\phi$ allows for global translations and dilations. Correspondingly, the quadratic polynomial in $g$ allows for the detection of any mismatch in the mean and co-variance of the two distributions. One can easily check that, with these basic functions available, the procedure yields the exact solution to the optimal transport problem between arbitrary Gaussians.

If these are the only features available, then there is no advantage in dividing the global problem into local ones, as the composition of linear maps is also linear, thereby providing no additional richness to the single step scenario. The natural element to add is an adaptive feature that could perform --and detect the need of-- local mass displacements. In one dimension, a natural choice is provided by one or more Gaussians of the form
$$ \phi_{nl}^k = \alpha_k \exp \left( - \frac{[v_k (x - \bar{x}_k)]^2}{2} \right), 
\quad g_{nl}^k = \beta_k \exp \left( - \frac{[s_k (z - \bar{z}_k)]^2}{2} \right), $$  
where the index $k$ labels the Gaussian feature when more than one is used.
The Gaussians in $\phi$ allow for local stretching/compression around $m$ with scale $|v|^{-1}$ and amplitude $\alpha$, while each Gaussian in $g$ detects local discrepancies between the two distributions, as opposed to the global scale and positioning provided by its quadratic component. The parameters $v$, $\bar{x}$, $s$ and $\bar{z}$ appear nonlinearly in $\phi$ and $g$, moving us away from the linear feature spaces of \cite{kuang2017sample} and into the realm of adaptability, as the parameters automatically select the location and scale of the changes required by the data.

There are at least four alternative ways to bring these Gaussian features to higher dimensions:
\begin{enumerate}
  \item Adopt general Gaussians of the form
  $$ \phi_{nl} = \alpha \exp \left( - \frac{\|V (x - \bar{x})\|^2}{2} \right) , $$
  with $\bar{x}$ a vector and $V$ a matrix (it is more convenient to write the Gaussian in terms of a general matrix $V$ in this way, rather than in terms of the inverse covariance matrix $C^{-1} = V^T V$, as we would need to require this to be positive definite);
  
  \item adopt isotropic Gaussians
  $$ \phi_{nl} = \alpha \exp \left( - \frac{v \|x - \bar{x}\|^2}{2} \right) , $$
  with $v$ a scalar,
  
  \item adopt one-dimensional Gaussians along arbitrary directions
  $$ \phi_{nl} = \alpha \exp \left( - \frac{\|v\cdot(x - \bar{x})\|^2}{2} \right) , $$
  with $v$ a vector, and
  
  \item adopt a Gaussian with diagonal covariance
  $$ \phi_{nl} = \alpha \exp \left( - \frac{\|D (x - \bar{x})\|^2}{2} \right) , $$
  with $D$ a diagonal matrix,
\end{enumerate} 
and similarly for $g_{nl}$ in all four cases. The first choice has the advantage of generality but may be prone to overfitting in high dimensions, unless it is severely penalized. The second approximates a general function $\phi$ by the composition of isotropic bumps, an appropriate image is that of hammering a sheet of metal into any desired shape. Yet, it would resolve poorly local, one-dimensional changes. The third choice excels at these but will fare poorly for more isotropic local changes. Finally, the fourth choice is attached to the coordinate axes, which would make sense only if these correspond to variables that are assumed to change independently.

A natural question is how many Gaussians to include in the functional space proposed. We have used two in the examples below, but one Gaussian would have sufficed: in the adversarial multistep method proposed, it is enough that the player with strategy $g(y)$ has a ``lens'' (the Gaussian) to identify the area where the two distributions least agree, and the player with strategy $\phi(x)$ has the capability to perform local moves to correct this misfit. Since the center and width of the Gaussian are free parameters, both assertions hold. With a single Gaussian feature, both players can focus only on one local misfit at a time. However, the algorithm has multiple steps, so effectively the total number of features available is the product of the features per step times the number of steps.

\subsection{Local Algorithm}\label{subsec:localAlgorithm}
We will use vectors $\alpha \in \mathbb{R}^a, \beta \in \mathbb{R}^b$ to parametrize $\phi(x) = \phi_{\alpha}(x)$ and $g(y) = g_{\beta}(y)$.
We are seeking to solve the minimax problem in $\alpha \in \R^a, \beta \in \R^b$ for the Lagrangian:
\[
L[\alpha,\beta] = \frac{1}{n} \sum_{i=1}^n g_{\beta}(\nabla \phi_{\alpha}(x_i)) - \frac{1}{m} \sum_{j=1}^m e^{g_{\beta}(y_j)} + P(\alpha,\beta)
\]
where $P$ is a penalization function that will be described in Section \ref{subsec:penalty}.

In practice, one could use any available minimax solver to find a critical point of the above Lagrangian. Yet, to our knowledge, there is no available efficient method suitable for a non-convex/non-concave landscape.

A naive algorithm would simultaneously implement gradient descent in $\alpha$ and  gradient ascent in $\beta$, with updates given at each step $s$ by:
\begin{align*}
\alpha^{s+1} &= \alpha^s - \eta \nabla_{\alpha}L[\alpha^s, \beta^s] \\
\beta^{s+1} &= \beta^s + \eta \nabla_{\beta}L[\alpha^s, \beta^s],
\end{align*}
with a step size $\eta$ that may change at each iteration. From a game-theory perspective, this corresponds to two myopic players that plan their next move based only on their current position, without anticipating what the other player might do.

%

Instead, more insightful players will choose their next move based on the future position of their opponents. This yields a second order algorithm, that we will refer to as \emph{implicit} gradient descent, with updates given by:
\begin{align*}
\alpha^{s+1} &= \alpha^s - \eta \nabla_{\alpha}L[\alpha^{s+1}, \beta^{s+1}]  \\
\beta^{s+1} &= \beta^s + \eta \nabla_{\beta}L[\alpha^{s+1}, \beta^{s+1}] .
\end{align*}
A simple Taylor expansion gives: 
\begin{align*}
\nabla_{\alpha}L[\alpha^{s+1}, \beta^{s+1}] &\approx  \nabla_{\alpha}L^s + \nabla^2_{\alpha \alpha} L^s \cdot (\alpha^{s+1}- \alpha^s) + \nabla^2_{\alpha \beta} L^s \cdot (\beta^{s+1}- \beta^s) \\
\nabla_{\beta}L[\alpha^{s+1}, \beta^{s+1}] &\approx  \nabla_{\beta}L^s + \nabla^2_{\alpha \beta} L^s \cdot (\alpha^{s+1}- \alpha^s) + \nabla^2_{\beta \beta} L^s \cdot (\beta^{s+1}- \beta^s)  
\end{align*}
Defining the \emph{twisted} gradient $G^s$ and \emph{twisted} Hessian $H^s$ by
\[
G^s = \begin{pmatrix} \nabla_{\alpha} L^s \\ - \nabla_{\beta} L^s \end{pmatrix}, \quad H^s = \begin{pmatrix} \nabla^2_{\alpha \alpha} L^s & \nabla^2_{\alpha \beta} L^s \\ - \nabla^2_{\alpha \beta} L^s & -\nabla^2_{\beta \beta} L^s \end{pmatrix}
\]
and $\gamma^s =\begin{pmatrix}   \alpha^s \\ \beta^s \end{pmatrix}$, one obtains the second-order updating scheme:
\begin{equation}
\gamma^{s+1} = \gamma^s - \eta \left(I + \eta H^s \right)^{-1} G^s
\end{equation}
Notice that as $\eta \to 0$, the scheme is equivalent to a classical gradient descent. On the other hand, as $\eta \to +\infty$, the scheme converges to Newton iterations.

At each iteration, we are allowed to update $\eta$ in order to accelerate convergence. Ongoing research \cite{minimax2018essid} addresses the correct rules to update $\eta$, as well as the convergence of the algorithm to a critical point of the Lagrangian.
This minimax solver is robust in two senses: it guarantees both convergence to a local minimax point and constant improvement. The latter has to do with the subtlety of minimax problems, as opposed to regular minimization where enforcing a decrease of the objective function is enough. In each step of our implicit procedure to $\min_x \max_y L(x,y)$, if $L[x^{s+1}, y^{s+1}]$ is either bigger than $L[x^s, y^{s+1}]$ or smaller than $L[x^{s+1}, y^s]$, we reject the step and adopt a smaller learning rate. Because of this, the solution will always improve over the starting identity map. 
If computing the twisted Hessian $H$ becomes too costly, one can resort to Hessian approximation techniques such as BFGS or its variations \cite{wright1999numerical,pavon2017variational}.

To conclude, the algorithm for finding the optimal match between two consecutive distributions, which we denote sample based local optimal transport (SBLOT), is summarized in Algorithm \ref{algo:SBLOT}.

\begin{algorithm}
    \caption{Sample Based Local Optimal Transport Algorithm (SBLOT)}
    \label{algo:SBLOT}
    \begin{algorithmic}[0] 
        \Procedure{SBLOT}{$(x_i),(y_j)$} 
        	\State Initialize $\gamma$
        	\State Compute the twisted gradients and Hessians $G, H$
			\For{$n=1,..,MaxIter$}
				\If{$||G|| < tolerance$}
					\State break
				\EndIf
				\State $\gamma \gets \gamma - \eta (I + \eta H)^{-1} G$ 
				\State Recompute the twisted gradients and Hessians $G, H$ at $\gamma$
				\State Update $\eta$
			\EndFor	
            \State \textbf{return} ${\displaystyle \nabla \phi_{\gamma[1:a]}(x)}$
        \EndProcedure
    \end{algorithmic}
\end{algorithm}


\subsection{Penalization}\label{subsec:penalty}
Transforming Problem \ref{minimax} into Problem \ref{sampleMinimax} amounts to replacing the theoretical measures with their empirical estimates;
\[
\rho \approx \hat{\rho} = \frac{1}{n} \sum_{i=1}^n \delta_{\nabla \phi(x_i)}, \quad \nu \approx \hat{\nu} = \frac{1}{m} \sum_{j=1}^m \delta_{y_j}
\]
Even if $\rho \ll \nu$, this will not hold for their estimates. Allowing maximum freedom for the function $g$ will result in an infinite Kullback-Leibler divergence. For instance, if one allows functions $g$ with support including some $\nabla \phi(x_i)$ but none of the $y_j$, the Lagrangian will grow unboundedly, since the exponential term that regularly inhibits this growth is now constant. One way to avoid this problem is to use the relative entropy not between $T(X)$ and $Y$ but between $T(X)$ and $(1-\epsilon) Y + \epsilon T(X)$, as then the law of $T(X)$ is always absolutely continuous w.r.t. the law of $(1-\epsilon)Y + \epsilon T(X)$, eliminating the possibility of blowup in $g$, and the minimum is still reached when $T(X)=Y$.
%
Another general simple way to avoid this kind of scenario is through the addition to the Lagrangian of terms that penalize overfitting. For our particular choice of functional spaces, it is only the coefficients in the argument of the exponentials that require penalization, as those are the only ones than involve spatial scales. In particular, for a component of $g$ or $\phi$ of the form
$$ a e^{-(b \cdot (x-c))^2} , $$
we add penalization terms proportional to
$$ e^{(\epsilon \|b\|)^2}, $$
with $\epsilon$ as defined above, to avoid resolving scales smaller than $\epsilon$, to 
$$ \frac{1}{(D\|b\|)^2}, $$
where $D$ measures the diameter of the support of the data, to avoid having Gaussians so broad that they are indistinguishable from the quadratic components of the functional space, to
$$ \left\|\frac{c}{D}\right\|^2,$$
to avoid centering the Gaussian away from the data, and, when more that one Gaussian is used, to
$$ \frac{\epsilon^2}{\|c_i-c_j\|^2}, $$
for every pair $(i,j)$ of Gaussians, to avoid possible degeneracies in the functional space when two Gaussians become nearly indistinguishable. 

All these terms are added and multiplied by a tunable parameter $\lambda$. Yet one more consideration is required for the penalization of the parameters of the potential $\phi$: since in the Lagrangian, $\phi$ appears only as an argument of $g$, for a fixed $\lambda$, the penalization terms and the core Lagrangian can easily become unbalanced. In particular, at the exact solution, $g$ is zero, so only the penalization terms will remain. To correct for such imbalance, we multiply the corresponding penalization terms by the average value of $\|\nabla g\|$ over all current $\nabla\phi(x_i)$.


\section{Experiments}\label{sec:experiments}

This section illustrates the algorithm through some simple examples. First we use a one-dimensional example --simplest for visualization-- and a direct solver between initial and final distributions to display the way in which the function $g$ adapts, creating features that point to those areas where transport in still deficient, thus guiding $\phi$ to correct them. The two distributions in the first example are relatively close, so that they can be matched without involving intermediate distributions. A second set of one-dimensional examples follows, involving more significant changes and hence requiring the use of interpolated distributions.  Then we perform some  two-dimensional examples, involving Gaussians, Gaussian mixtures and a distribution uniform within an annulus. Finally, we use an example built so that we know the exact answer, to perform an empirical analysis of convergence. All the examples presented are intended for illustration and use synthetic data; applications to real data, particularly to change detection, will be presented in field-specific articles currently under development.

\subsection{Adversarial behavior of $\phi$ and $g$}

This section shows, through a simple experiment, the competitive behavior exhibited by the two players $\phi$ and $g$ in the local algorithm (Algorithm \ref{algo:SBLOT}). To this end, we create data where the initial and final distribution are not very far from each other, so that the local algorithm can be used as a stand alone routine. More specifically, we map one single Gaussian distribution to a Gaussian mixture, where the two components of the mixture overlap significantly, so that they do not differ too markedly from the source.

\begin{figure}[h]
\begin{minipage}{0.32\textwidth}
\centering
\includegraphics[width=\textwidth]{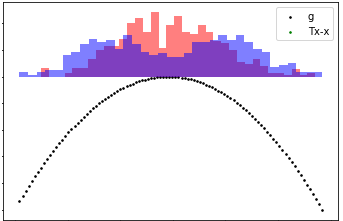}\\
Iteration 0
\end{minipage}
\begin{minipage}{0.32\textwidth}
\centering
\includegraphics[width=\textwidth]{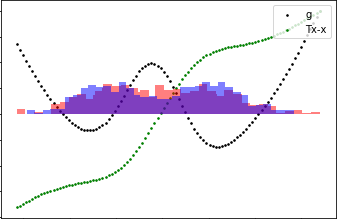}\\
After $8$ iterations
\end{minipage}
\begin{minipage}{0.32\textwidth}
\centering
\includegraphics[width=\textwidth]{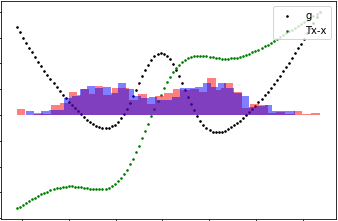}\\
After $17$ iterations
\end{minipage}
\caption[Plot at three different iteration times of Algorithm \ref{algo:SBLOT}]{Plot at three different iteration times of Algorithm \ref{algo:SBLOT}. Histograms of the source samples and their transforms are in red, and of the target samples in blue. The black curve corresponds to $g(x)$, vertically rescaled for visualization. The green curve represents the displacement $T(x)-x$.}
\label{fig:Adversarial}
\end{figure}

Figure \ref{fig:Adversarial} shows steps in the solution to the corresponding sample based OT problem, with the source samples $(x)_i$ from a Gaussian --and their transforms-- in red and the samples $(y)_j$ from a mixture of two Gaussians in blue.
Point samples are represented through histograms. The figure on the left represents the initial configuration, the one in the middle the configuration after 10 iterations of Algorithm \ref{algo:SBLOT}, and the one on the right the final configuration.

On top of the histograms, we display the function $g(x)$ in black, scaled vertically to be in the interval $[-1;1]$ for easier comparison with the data, and the displacement $\nabla \phi(x)-x$ in green,  representing the map that sends the initial sample (in red, in the left figure) to the current sample (in red, in the middle or right figure).

The initial displacement, being $0$, was not represented at initialization, but we initialize the function $g(z)$ at the purely quadratic function:
\begin{equation}\label{eq:ggauss}
\frac{1}{2} z^T \left( \hat{\Sigma}_y^{-1} - \hat{\Sigma}_x^{-1} \right) z + \left( \hat{\Sigma}_x^{-1} \hat{x} - \hat{\Sigma}_y^{-1} \hat{y} \right)^T z + \frac{1}{2} \left( \hat{y}^T \hat{\Sigma}_y^{-1} \hat{y} - \hat{x}^T \hat{\Sigma}_x^{-1} \hat{x}\right)
\end{equation}
where $\hat{x}, \hat{y}$ are the empirical means of the samples $(x)_i,(y)_j$, and $\hat{\Sigma}_x,\hat{\Sigma}_y$ their empirical covariance matrices.
Equation \ref{eq:ggauss} represents the optimal $g$ for two Gaussian measures. More generally, starting with this expression as the initial guess for $g$ instructs $\phi$ to shift the samples as well as to stretch/compress them, in order to match the first and second moments of the two distributions.

The left image of Figure \ref{fig:Adversarial} shows how $g$ highlights the lack of variance in $(x)_i$; its maximum is at $0$, and it has smaller values at the edges. This forces $\phi$ to adapt accordingly, by applying a linear map to stretch $(x)_i$. When the variance of the $(\nabla \phi(x))_i$ exceeds the variance of the $(y)_j$, the shape of $g$ is inverted.

In the middle image of Figure \ref{fig:Adversarial}, we can see that $\nabla \phi$ corrected the mismatches highlighted by $g$ and even started to slightly separate the mass in the middle. However, there is still too much red mass around $0$ and too little red mass around the two peaks of the blue Gaussian mixture. 
This is well detected by $g$, which has a local maximum within the area of red mass excess and two local minima within the area of red mass default.
In the right image of Figure \ref{fig:Adversarial}, we observe that $\nabla \phi$ adapted accordingly and starts yielding satisfactory results. At this point, $g$ is very close to $0$ ($||g||_{\infty}\sim 10^{-5}$), although this is not apparent in the figure due to the normalization we applied for plotting.

\subsection{The global algorithm in dimension one}

Figures \ref{fig:gto2g} and \ref{fig:gto3g} represent inputs and outputs of Algorithm \ref{algo:SBOT}, where $(x)_i$ is sampled from a Gaussian and $(y)_j$ from a mixture of two and three Gaussians respectively.
\begin{figure}[h]
\begin{center}
\begin{minipage}{0.35\textwidth}
\centering
\includegraphics[width=\textwidth]{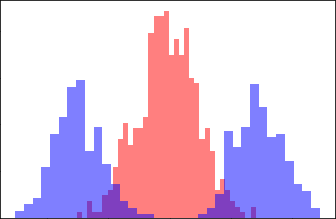}\\
Starting configuration
\end{minipage}
\begin{minipage}{0.35\textwidth}
\centering
\includegraphics[width=\textwidth]{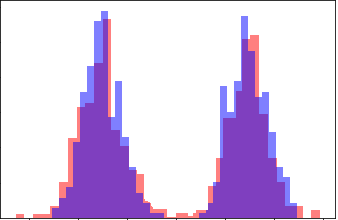}\\
Final configuration
\end{minipage}
\end{center}
\caption{Algorithm \ref{algo:SBOT} pushing forward a Gaussian to a mixture of two Gaussians, in 1D. The source samples and their transforms are depicted through histograms in red, and the target samples in blue.}
\label{fig:gto2g}
\end{figure}

\begin{figure}[h]
\begin{center}
\begin{minipage}{0.35\textwidth}
\centering
\includegraphics[width=\textwidth]{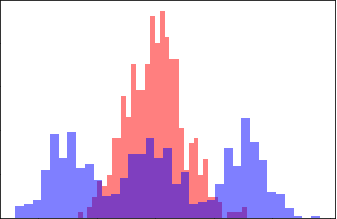}\\
Starting configuration
\end{minipage}
\begin{minipage}{0.35\textwidth}
\centering
\includegraphics[width=\textwidth]{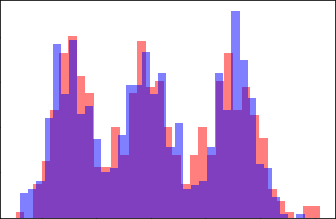}\\
Final configuration
\end{minipage}
\end{center}
\caption{Same as figure \ref{fig:gto2g} but with a mixture of three Gaussians as target.}
\label{fig:gto3g}
\end{figure}

These results were obtained by generating $\sim 200$ samples for the source and target measures, and using the functional spaces defined in Section \ref{subsec:functionalspace}  in the local algorithm (Algorithm \ref{algo:SBLOT}), with a general quadratic form for both $\phi$ and $g$, plus one adaptive Gaussian for $\phi$ and two for $g$. A total of $N=10$ and $N=20$ intermediary measures were adopted for the first and second example, respectively. As one can see, even though each local map can only perform one local deformation, the composition of many creates all the complexity required to move one single Gaussian to a mixture of two or three.

\subsection{Two-dimensional examples}

Switching to two dimensions, 
Figure \ref{fig:gtoa} represents the results of mapping a Gaussian distribution to a uniform distribution within an annulus.

An isotropic Gaussian was used for $\phi_{nl}$ and two for $g_{nl}$ in the functional space of Algorithm \ref{algo:SBLOT}, and $N=30$ intermediary distributions were used in Algorithm \ref{algo:SBOT}. 
%
Figure \ref{fig:gtoadisp} represents the displacement interpolants at $t=k/5$ for $ k=1,\ldots ,5$, obtained from running Algorithm \ref{algo:SBOT} on the example in Figure \ref{fig:gtoa}. In addition to mass spreading from the isotropic Gaussian, the linear and quadratic part of $\phi$ translated and stretched the red sample accordingly.

\begin{figure}[hb!]
\begin{center}
\begin{minipage}{0.35\textwidth}
\centering
\includegraphics[width=\textwidth]{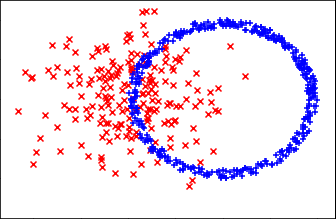}\\
Starting configuration
\end{minipage}
\begin{minipage}{0.35\textwidth}
\centering
\includegraphics[width=\textwidth]{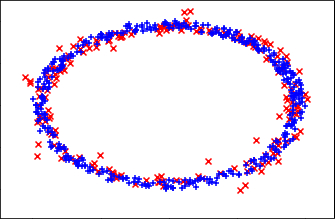}\\
Final configuration
\end{minipage}
\end{center}
\caption{Algorithm \ref{algo:SBOT} from a displaced Gaussian to an annulus, in 2D}
\label{fig:gtoa}
\end{figure}

\begin{figure}[hb!]
\begin{center}
\includegraphics[scale=0.6]{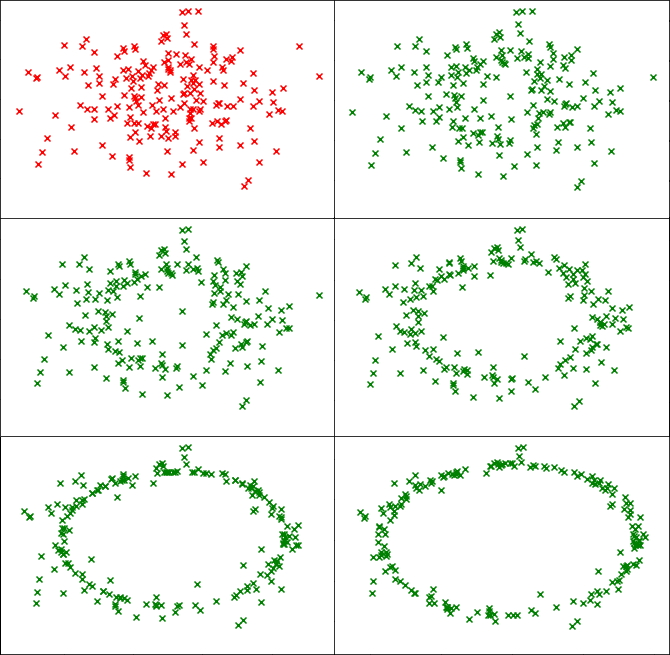}
\end{center}
\caption[Interpolants given by Algorithm \ref{algo:SBOT} from a Gaussian to an annulus, in 2D]{Interpolants given by Algorithm \ref{algo:SBOT} from a Gaussian to an annulus, in 2D. The top left figure (red) corresponds to the original sample. Time flows from left to right, and from top to bottom. Subsequently represented are the interpolants at time $t=k/5$ for $k=1,\ldots,5$. }
\label{fig:gtoadisp}
\end{figure}

Similarly, Figure \ref{fig:gto2g2} represents the initial and final configurations obtained from running Algorithm \ref{algo:SBOT} to transport a two-dimensional Gaussian distribution to a mixture of two Gaussians.
A diagonal covariance was used in the non-linearity $\phi_{nl}$ for the functional space in Algorithm \ref{algo:SBLOT}, and $N=30$ intermediary steps were used in Algorithm \ref{algo:SBOT}. This type of non-linearity is well adapted to separate samples along the horizontal and vertical axes.

Figure \ref{fig:gto2gdisp} represents the displacement interpolants at $t=k/5$ for $ k=1,\ldots ,5$, obtained from running Algorithm \ref{algo:SBOT} on the example in Figure \ref{fig:gto2g2}.

\begin{figure}[htb!]
\begin{center}
\begin{minipage}{0.35\textwidth}
\centering
\includegraphics[width=\textwidth]{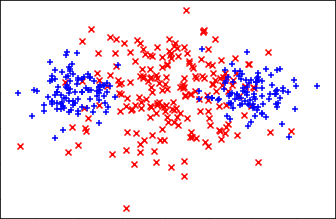}\\
Starting configuration
\end{minipage}
\begin{minipage}{0.35\textwidth}
\centering
\includegraphics[width=\textwidth]{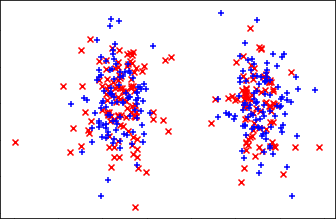}\\
Final configuration
\end{minipage}
\end{center}
\caption{Algorithm \ref{algo:SBOT} from a Gaussian to a mixture of 2 Gaussians, in 2D}
\label{fig:gto2g2}
\end{figure}

\begin{figure}[htb!]
\begin{center}
\includegraphics[scale=0.6]{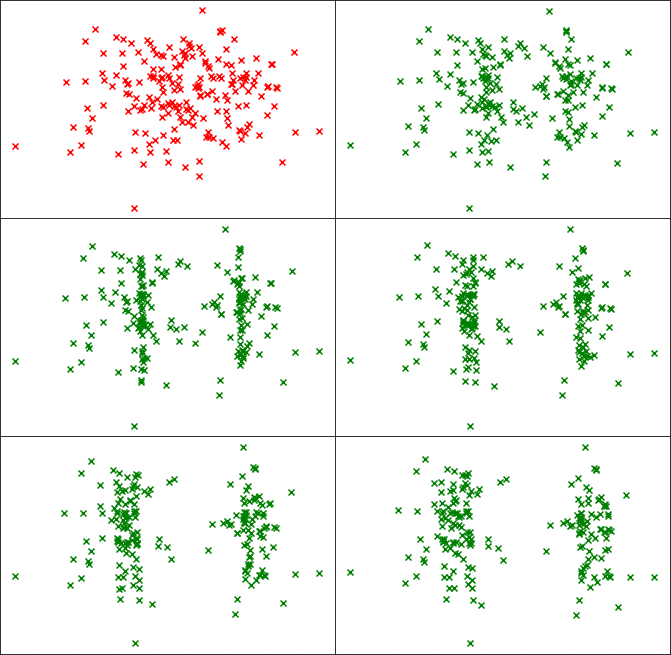}
\end{center}
\caption[Interpolants given by Algorithm \ref{algo:SBOT} from a Gaussian to a mixture of two Gaussians, in 2D]{Interpolants given by Algorithm \ref{algo:SBOT} from a Gaussian to a mixture of two Gaussians, in 2D. The top left figure (red) corresponds to the original sample. Time flows from left to right, and from top to bottom. Subsequently represented are the interpolants at time $t=k/5$ for $k=1,\ldots,5$. }
\label{fig:gto2gdisp}
\end{figure}

\subsection{Empirical analysis of convergence}

In this subsection, we empirically analyze the convergence of the algorithm in a situation where the generating distributions, as well as the optimal map, are known: $(x_i)_{i=1,\cdots,n}$ are i.i.d. samples of a standard Gaussian distribution, $(y_j)_{j=1,\cdot,m}$ are obtained through $y_i = \phi'(x_i)$ for $\phi(x) = |x|^{1+\epsilon}$ ($\epsilon= 1/4$). Brenier's theorem guarantees that, since $\phi$ is convex, $\phi'$ is the optimal map for the quadratic Wasserstein problem.

In a first set of experiments, we keep the number of samples constant at $n=m=500$, and we vary the number of intermediary steps $K$ in the global algorithm, raging through $K=1,2,3,5,10$.
In a second set of experiments, we keep the number of intermediary steps in the global algorithm constant at $K=10$, and vary the number of sample points, using $n=m=25,50,100,200,500$.
In both sets, we compute the experimental map $\nabla \phi_{exp}$ by \eqref{eq:compMaps}, and compare it to the optimal $\nabla \phi^*$ defined by:
\[
\nabla \phi^* (x) = (1 + \epsilon) x |x|^{\epsilon - 1}.
\]

In each experiment, three numerical quantities are computed:
\begin{enumerate}
\item The weighted $L^2$ norm ${\displaystyle \int |\nabla \phi_{exp}(x) - \nabla \phi^*(x)|^2 \mu(x)dx} \approx \sum_i |\nabla \phi_{exp}(x_i) - \nabla \phi^*(x_i)|^2 $,
\item The $L^{\infty}$ norm between $\nabla \phi_{exp}$ and $\nabla \phi^*$,
\end{enumerate}

For illustrative purposes, we show in Figure \ref{fig:compK} the differences between $\nabla \phi_{exp}$ and $\nabla \phi^*$ for various sets of parameters.
Tables \ref{table:varK} and \ref{table:varN} summarize the results.

\begin{table}[H]
\begin{center}
\begin{tabular}{|l|l|l|l|l|l|}
\hline
$K=$ & \multicolumn{1}{l|}{1} & \multicolumn{1}{l|}{2} & \multicolumn{1}{l|}{3} & \multicolumn{1}{l|}{5} & \multicolumn{1}{l|}{10} \\ \hline
$\E[|\nabla \phi^*(X) - \nabla \phi_{exp}(X)|^2]$ & 0.74 & 0.55 & 8.3 $\cdot 10^{-1}$   & 1.7 $\cdot 10^{-2}$  & 8.7 $\cdot 10^{-3}$  \\ \hline
$||\nabla \phi^* - \nabla \phi_{exp}||_{L^{\infty}}$ & 0.53  & 0.22  & 9.9 $\cdot 10^{-2}$  &  8.7 $\cdot 10^{-2}$  & 6.2 $\cdot 10^{-2}$  \\ \hline
\end{tabular}
\end{center}
\caption{Convergence as a function of the number $K$ of intermediary steps}
\label{table:varK}
\end{table}

\begin{table}[H]
\begin{center}
\begin{tabular}{|l|l|l|l|l|l|}
\hline
$n=m=$ & \multicolumn{1}{l|}{25} & \multicolumn{1}{l|}{50} & \multicolumn{1}{l|}{100} & \multicolumn{1}{l|}{200} & \multicolumn{1}{l|}{500} \\ \hline
$\E[|\nabla \phi^*(X) - \nabla \phi_{exp}(X)|^2]$ & 1.4 & 0.35 & 7.1 $\cdot 10^{-2}$  & 2.1 $\cdot 10^{-2}$ & 8.7 $\cdot 10^{-3}$ \\ \hline
$||\nabla \phi^* - \nabla \phi_{exp}||_{L^{\infty}}$ & 1.3  & 0.49  & 0.16  &  0.11 & 6.2 $\cdot 10^{-2}$  \\ \hline
\end{tabular}
\end{center}
\caption{Convergence as a function of the number of samples $n$}
\label{table:varN}
\end{table}

In practice, setting a number of samples less than $15$ in this example leads to poor convergence due to the extreme sparsity of data.

\begin{figure}[h]
\begin{minipage}{0.245\textwidth}
\centering
\includegraphics[width=\textwidth]{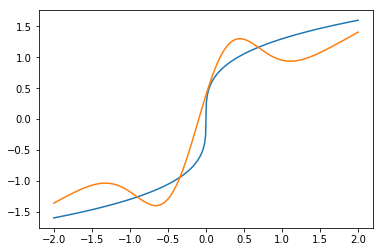}\\
$K=1$
\end{minipage}
\begin{minipage}{0.245\textwidth}
\centering
\includegraphics[width=\textwidth]{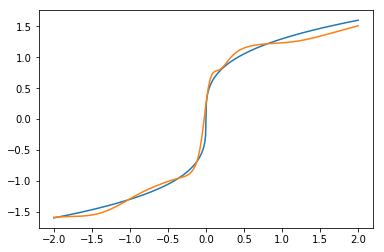}\\
$K=3$
\end{minipage}
\begin{minipage}{0.245\textwidth}
\centering
\includegraphics[width=\textwidth]{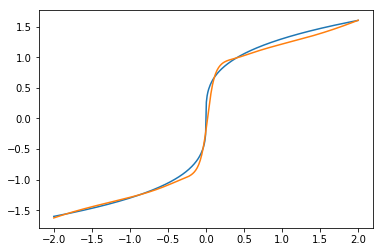}\\
$K=5$
\end{minipage}
\begin{minipage}{0.245\textwidth}
\centering
\includegraphics[width=\textwidth]{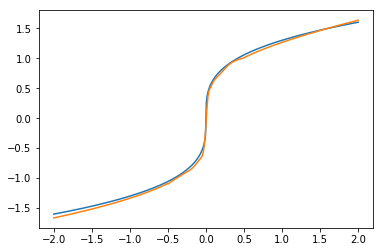}\\
$K=10$
\end{minipage}
\caption[Comparison between $\nabla \phi^*$ and $\nabla \phi_{exp}$ for different values of $K$]{Comparison between $\nabla \phi^*$ (blue) and $\nabla \phi_{exp}$ (orange) for different values of intermediary steps $K$.}
\label{fig:compK}
\end{figure}

Figure \ref{fig:compK} compares the optimal map $\nabla \phi^*$ with the computed map $\nabla \phi_{exp}$. Note that the one step algorithm does not provide a monotone solution, i.e. it is not the gradient of a convex function: the source and target distributions are not close enough to guaranty that. This is corrected through the introduction of intermediate steps, which brings the source and target distributions for each step closer to each other via displacement interpolation. For the example under consideration, the optimal solution is convex for any value of $K$ bigger than $4$. Notice also that, for $K=10$ and $n=500$, the solution approximates the exact one very accurately in the bulk of the distribution, as captured by the density-weighted $L_2$ norm of their difference. On the other hand, the $L_{\infty}$ norm is dominated by the behavior at the tails, where little data is present to guide the algorithm.


\section{Discussion and Conclusions}\label{sec:conclusion}

We have developed an adaptive methodology for the sample-based optimal transport problem under the standard quadratic cost function. The main advantage of the new procedure is that it does not require any external input on the form of the distributions that one seeks to match, or any expert knowledge on the type, location and size of the features in which the source and target distribution may differ. 

Even though the map $\nabla \phi$ and test function $g$ used at each step are parametric, by using the composition of many simple maps and having at one's disposal a ``lens'' within $g$ that can focus on any individual local mismatch at each step, the resulting procedure can be thought as effectively free of parameters, except for the number of intermediate distributions to use, a stopping criterion, and a couple of constants associated with the penalization of the nonlinear features. Thus, it has the potential to form the basis for a universal tool that can be transferred painlessly across fields. 

Two main ingredients allow for the procedure to capture arbitrary variability without making use of a huge dictionary of candidate features (in its current version, it uses only three: a linear feature for global displacements, a quadratic feature for global scalings, and a Gaussian feature for localized displacements). One ingredient, borrowed from prior work in \cite{kuang2017sample}, is the factorization of the potentially quite complex global map into a sequence of much simpler local maps between nearby distributions. The optimality of the composed map is guaranteed through the use of displacement interpolation. The second ingredient is the formulation of the local problem as a two-player game where the first player seeks to push forward one distribution into the other, while the second player develops features that show where the push-forward condition fails. The variational characterization of the relative entropy between distributions that gives rise to this game-theory formulation has the additional advantage of being sample-friendly, as it involves the two distributions only through the expected values of functions, which can be naturally replaced by empirical means. Because the map between any two consecutive distributions is close to the identity, local optimality is guaranteed by requiring this map to be the gradient of a potential.

Topics for future research include the extension of the algorithm to transportation costs different from the squared distance and, for the purpose of more efficient computability, the optimization of the minimax solver and the parallelization of the computation of the local maps. Most of all, we believe, the use of the new methodology in real applications will shed light on the issues that require further work, which may include the development of features and penalizations suitable for efficiently capturing sharp edges or removed objects.

%
%
%
%

\begin{center}
\section*{\small Acknowledgments}
\end{center}
The authors would like to thank Yongxin Chen for connecting our variational formulation of the Kullback-Leibler divergence with the Donsker-Varadhan formula. This work has been partially supported by a grant from the Morse-Sloan Foundation. The work of Tabak was partially supported by NSF grant DMS-1715753 and ONR grant N00014-15-1-2355. The work of Essid was partially supported by NSF grant DMS-1311833.
\bigskip
\bibliographystyle{spmpsci}      
\bibliography{adaptiveot}   


\end{document}